\documentclass[a4paper, 11pt]{amsart}

\usepackage[english]{babel} 
\usepackage{amsmath, amssymb} 
\usepackage[all]{xy} 
\usepackage{mathtools}

\usepackage{hyperref}

\renewcommand{\l}[1]{%
    \begingroup 
    \mathtoolsset{
    prescript-sup-format=\mathit,%
    }%
    \prescript{}{l}{#1}%
    \endgroup%
    }

\newcommand{\loccit}{\textit{loc.~cit.~}}

\newcommand{\cat}[1]{\mathrm{#1}}
\newcommand{\La}{\cat{La}}
\newcommand{\Sh}{\cat{Sh}}
\newcommand{\Set}{\cat{Set}}

\newcommand{\Q}{\mathbf{Q}}

\newcommand{\N}{\mathbf{N}}

\newcommand{\R}{\mathbf{R}}

\newcommand{\cA}{\mathcal{A}}
\newcommand{\cB}{\mathcal{B}}
\newcommand{\cG}{\mathcal{G}}
\newcommand{\cT}{\mathcal{T}}
\newcommand{\cF}{\mathcal{F}}
\newcommand{\cX}{\mathcal{X}}

\newcommand{\sx}{\mathsf{x}}

\renewcommand{\lg}{\mathfrak{g}}

\newcommand{\ad}{\mathrm{ad}}
\newcommand{\an}{\mathrm{an}}
\newcommand{\lc}{\mathrm{lc}}

\DeclareMathOperator{\Hom}{Hom}
\DeclareMathOperator{\iHom}{\mathcal{H}\mathit{om}}
\DeclareMathOperator{\Aut}{Aut}

\DeclareMathOperator{\id}{id}
\DeclareMathOperator{\pr}{pr}

\DeclareMathOperator*{\colim}{colim}

\newcommand{\del}{\partial}

\newcommand{\ul}[1]{\underline{#1}}

\renewcommand{\dot}{\bullet}

\newtheorem*{thm*}{Theorem}
\newtheorem{cor}[equation]{Corollary}
\newtheorem*{cor*}{Corollary}
\newtheorem{lemma}[equation]{Lemma}
\newtheorem{prop}[equation]{Proposition}

\theoremstyle{definition}
\newtheorem{dfn}[equation]{Definition}
\newtheorem*{dfn*}{Definition}

\theoremstyle{remark}
\newtheorem{rem}[equation]{Remark}
\newtheorem*{rem*}{Remark}

\newtheorem*{ex*}{Example}
\newtheorem{ex}[equation]{Example}

\title{On an analytic version of Lazard's isomorphism}
\author{Georg Tamme}
\address{Fakult\"at f\"ur Mathematik, Universit\"at Regensburg, 93040 Regensburg, Germany}
\email{georg.tamme@ur.de}
\thanks{The author is supported by the CRC 1085 \emph{Higher Invariants} (Universit\"at Regensburg)}

\date{\today}

\begin{document}

 \begin{abstract}
 We prove a comparison theorem between locally analytic group cohomology and Lie algebra cohomology for locally analytic representations 
of a Lie group over a nonarchimedean field of characteristic 0. The proof is similar to that of van-Est's isomorphism and uses only a minimum of functional analysis.
 \end{abstract}

\maketitle

\tableofcontents

\section*{Introduction}

In his seminal paper \cite{Lazard} Lazard established two basic theorems concerning the cohomology
of a compact $\Q_{p}$-analytic Lie group $G$ with Lie algebra $\lg$. Firstly, if $V$ is a finite dimensional
$\Q_{p}$-vector space with continuous $G$-action, the natural map from locally analytic group
cohomology $H^{*}_{\an}(G,V)$, defined in terms of locally analytic cochains, to continuous group
cohomology $H^{*}_{\mathrm{cont}}(G,V)$ is an isomorphism. Secondly, there is a natural isomorphism between the
direct limit $\colim_{G'\subset G} H^{*}_{\mathrm{cont}}(G', V)$, where $G'$ runs through the system of open
subgroups of $G$, and the Lie algebra cohomology $H^{*}(\lg,V)$. Hence, combining both, there is a
natural isomorphism
\begin{equation}\label{eq:lazard}
\colim_{G'\subset G \text{ open}} H^*_{\an}(G',V) \cong H^*(\lg,V).
\end{equation}

These results play an important role in arithmetic geometry, in particular in the theory of Galois
representations, or in the study of $p$-adic regulators \cite{Huber-Kings}. 

At least for certain
Lie groups, integral and $K$-analytic versions have been obtained by Huber, Kings, and Naumann
\cite{HKN}, when $K$ is a finite extension of $\Q_p$.
The proofs are based on Lazard's original argument via continuous group cohomology, and are not easily accessible.
A somewhat simplified proof has been given by Lechner \cite{Lechner}
using formal group cohomology.

On the other hand, the situation for a real Lie group $G$ is much more transparent.
The analogous result is van Est's isomorphism $H^{*}_{d}(G,V) \cong
H^{*}(\lg,K;V)$, which relates differentiable group cohomology with relative Lie algebra cohomology
for a maximal compact subgroup $K \subseteq G$. 
Its proof is based on the following observations: The quotient
$G/K$ is contractible, hence the de Rham complex $\Omega^{*}(G/K,V)$ with coefficients in a
$G$-representation $V$ is a resolution of $V$. Moreover, for any $k$, the space $\Omega^{k}(G/K,V)$
is $G$-acyclic. Hence, $H^{*}_{d}(G,V)$ is computed by the $G$-invariants of the complex
$\Omega^{*}(G/K,V)$, which is precisely the Chevalley-Eilenberg complex computing relative Lie
algebra cohomology $H^{*}(\lg,K;V)$.

It is a natural question whether a similar argument works in the nonarchimedean situation. 
In this note, we show that this is indeed the case. This gives a direct proof of the
isomorphism \eqref{eq:lazard}  and generalizes it with respect to the ground field and the coefficients:
\begin{thm*}
Let $K$ be a nonarchimedean field of characteristic $0$.
Let $G$ be a locally $K$-analytic Lie group and $G\to \Aut(V)$ a locally analytic representation on a barrelled locally convex $K$-vector space.
Denote by $\lg$ the $K$-Lie algebra of $G$. Then there are natural isomorphisms
\[
\colim_{G'\subset G \text{ open}} H^{*}_{\an}(G',V) \cong H^{*}(\lg, V),
\]
where the colimit is taken over the system of open subgroups of $G$.
\end{thm*}
The rough argument is as follows: The de Rham complex $\Omega^*(G,V)$ is a resolution of the locally constant $V$-valued functions on $G$. As in  the real case, each $\Omega^k(G,V)$ is $G$-acyclic, hence the cohomology of the locally constant $V$-valued functions on $G$ is isomorphic to the Lie algebra cohomology $H^*(\lg, V)$ (see Sections \ref{sec:DifferentialFormsAndLiaAlgebraCohomology} and \ref{sec:proof} for  precise results). The Theorem then follows by taking the direct limit over the open subgroups of $G$.

The proof also shows that, for compact  $G$, one can recover the locally analytic group cohomology from the Lie algebra cohomology as the invariants under the natural  $G$-action: $H^{*}_{\an}(G,V) \cong H^{*}(\lg, V)^{G}$ (see Corollary \ref{cor:LieAlgcohomforcompactG}).

Moreover, we describe the comparison map between locally analytic group cohomology and Lie algebra cohomology explicitly on the level of complexes: It is given by differentiating locally analytic cocycles at $1$ (see Section \ref{sec:explicit}).
As pointed out by the referee, one can use the methods of \cite{Huber-Kings} to show that, on cohomology groups, this comparison map agrees with the one studied by Lazard in the case that $K$ is $\Q_{p}$ and $V$ is finite dimensional.

In order to apply usual arguments from homological algebra, we show, following Flach \cite{Flach}, that the locally analytic cochain cohomology groups can be interpreted as derived functors of the global section functor on a topos $BG$ (Sections \ref{sec:1} and \ref{sec:2}). The nice feature of this is that it gives a quick proof of the main results which requires only a minimum of functional analysis.

An alternative approach to the cohomology of locally analytic representations of Lie groups over finite extensions of $\Q_{p}$  is due to Kohlhaase \cite{Kohlhaase}. It is based on relative homological algebra. 
He obtains similar results under an additional assumption on the group, which, as he proves, is fulfilled in many cases. 
The cohomology groups he defines are finer than ours in the sense that they themselves carry a locally convex topology. In contrast to the groups we use, they do not always coincide with the cohomology groups defined in terms of locally analytic cochains. 

I would like to thank the referee for useful remarks, in particular concerning the comparison of our isomorphism with Lazard's original one.

\subsection*{Notations and conventions}

For the whole paper, we let $K$ be a nonarchimedean field of characteristic $0$, i.e.,  $K$ is equipped with a nontrivial nonarchimedean absolute value $|\,.\,|$ such that $K$ is complete for the topology defined by $|\,.\,|$. 
By a manifold we will always mean a
paracompact, finite dimensional locally $K$-analytic manifold.
Note that, by \cite[Cor.~18.8]{Schneider}, any locally $K$-analytic Lie group is paracompact.
For manifolds $X,Y$, we denote by $C^{\an}(X,Y)$ the set of locally $K$-analytic maps from $X$ to
$Y$. We will refer to them simply as  analytic maps.

\section{Locally analytic group cohomology}
\label{sec:1}

In this section, we describe the topos-theoretic approach to locally analytic group cohomology. We refer to \cite{Flach} for the case of continuous cohomology.

Denote by $\La$ the category of  manifolds.  
We let $\Sh(\La)$ be the category
of sheaves on $\La$ with respect to the
topology generated by open coverings.
For this topology, every representable presheaf is a sheaf, hence we have the Yoneda embedding
$y\colon \La \to \Sh(\La)$.

Let $G$ be a Lie group. Then $y(G)$ is a group object in $\Sh(\La)$.
The category of sheaves with a $y(G)$-action is a topos
 \cite[Exp.~IV, 2.4]{SGA41}, 
called the classifying topos of $y(G)$. It will be denoted by $BG$.\footnote{More precisely, we assume the existence of universes and only consider manifolds which are elements of a given universe $\mathcal{U}$. Then $\Sh(\La)$ and $BG$ are $\mathcal{V}$-topoi for a universe $\mathcal{V}$ with $\mathcal{U} \in \mathcal{V}$.}
We denote its global section functor by 
$\Gamma\colon BG \to \Set, \Gamma(\cF)= \Hom_{BG}(*,\cF) = \cF(*)^G$.
Similarly, if $X$ is an object of $BG$, we denote by 
$\Gamma(X,-)=\Hom_{BG}(X,-)$ the functor of sections over $X$.
As usual, we define cohomology groups via the derived functors of the global section functor:
\begin{dfn}\label{def:general-def-of-cohom}
Let $\cA$ be an abelian group object of $BG$. Then we define
\[
H^i(G,\cA) := (R^i\Gamma)(\cA).
\]
\end{dfn}
\begin{ex}\label{ex:examples-of-sheaves}
Let $V$ be a finite dimensional $K$-vector space with a linear $G$-action such that 
the map $G\times V \to V$ defining the action is analytic. 
This induces an action $y(G)\times y(V) \to y(V)$, and hence $y(V)$ can naturally be considered as an element of $BG$. We have $\Gamma(y(V)) = V^{G}$. In the next section, we will show that the higher cohomology groups $H^{i}(G,y(V))$ coincide with the cohomology groups defined in terms of locally analytic cochains with coefficients in $V$.

As another example, let $M$ be a continuous $G$-module, i.e., a topological abelian group equipped with a linear $G$-action such that $G \times M \to M$ is continuous.
Then we have the sheaf of continuous $M$-valued functions $C(-,M)$ on $\La$. It also carries a natural action by $y(G)$. It follows from Proposition \ref{prop:cochains-neu} below that the groups $H^{i}(G, C(-,M))$ can be identified with the continuous cochain cohomology groups of $M$.
\end{ex}

We want to describe the cohomology groups defined in Definition~\ref{def:general-def-of-cohom} in terms of a concrete complex. 
We begin with some general considerations.

Let $\cT$ be a topos, and let $\cG$ be a group object in $\cT$.
For objects $\cA, \cB$ of $B\cG$ the internal hom $\iHom_{B\cG}(\cA,\cB)$ is given as follows: The underlying object of $\cT$ is
$\iHom_{\cT}(\cA,\cB)$ and the action of $\cG$ is given by the formula
$(g\phi)(a)=g(\phi(g^{-1}a))$.

Denote by $i\colon *\to \cG$ the morphism from the trivial group in $\cT$ to $\cG$. It induces a
geometric
morphism of topoi (see \cite[Exp.~IV, 4.5]{SGA41})
\[
i\colon \cT \cong B* \to B\cG.
\]
The left adjoint $i^*$ simply forgets the $\cG$-action.
The right adjoint is given by $i_*(\cF) = \iHom_{B\cG}(\l\cG, \cF)$ where $\l\cG$ is $\cG$ with its
natural left action, viewed as an object ob $B\cG$, and $\cF$ is viewed as object of $B\cG$ with
trivial $\cG$-action. The functor $i^{*}$ also has a left adjoint $i_{!}$ given by $\cF \mapsto \l\cG\times \cF$ with $\cG$-action via the first factor.

For an object $\cA \in B\cG$, we denote by $\cA^{\natural}$ the object of $B\cG$ with the same
underlying object in $\cT$ and trivial $\cG$-action.
\begin{lemma}\label{lem:iHom}
For $\cA, \cB \in B\cG$ we have
\[
\iHom_{B\cG}(\l\cG\times \cA, \cB) \cong  
i_*\iHom_\cT(i^*\cA, i^*\cB).
\]
\end{lemma}
\begin{proof}
Let $\cX$ be an object of $B\cG$. Then we have natural isomorphisms
\begin{align*}
\Hom_{B\cG}(\cX, i_{*}\iHom_\cT(i^*\cA, i^*\cB)) &\cong \Hom_{\cT}(i^{*}\cX, \iHom_\cT(i^*\cA, i^*\cB)) \\
&\cong \Hom_{\cT}(i^{*}(\cX\times \cA), i^{*}\cB) \\
&\cong \Hom_{B\cG}(i_{!}i^{*}(\cX\times\cA), \cB) \\
&\cong \Hom_{B\cG}(\l\cG\times (\cX\times \cA)^{\natural}, \cB) \\
&\cong \Hom_{B\cG}(\l\cG\times \cX\times \cA, \cB) \\
&\cong \Hom_{B\cG}(\cX, \iHom_{B\cG}(\l\cG\times \cA, \cB))
\end{align*}
where we used the isomorphism $\l\cG\times (\cX\times \cA)^{\natural} \cong \l\cG\times \cX\times \cA$ given by $(\pr_1,
\text{action})$. This implies the lemma.
\end{proof}

We now consider the case $\cT=\Sh(\La), \cG=y(G)$. For a sheaf $\cF$ on $\La$,
the sheaf underlying $i_{*}\cF$ is, by the above,  given by 
$X\mapsto \iHom_{\Sh(\La)}(y(G),\cF)(X) \cong \cF(G\times X)$ (Yoneda lemma). 
\begin{rem}\label{rem:strictly-paracompact}
By our general assumption, every manifold $X$ in $\La$ is paracompact. 
By \cite[Prop.~8.7]{Schneider}, it is even \emph{strictly paracompact}, i.e., every open covering of $X$ can be refined by a covering by pairwise disjoint open subsets.
This implies in particular that the functor of sections over $X$ is exact on the category of abelian sheaves on $\La$.
\end{rem}

\begin{lemma}
The functor $i_*$ from abelian sheaves on $\La$ to abelian group objects in $BG$ is exact.
\end{lemma}
\begin{proof}
Since $i_*$ is a right adjoint, it is left exact. Consider an epimorphism $\cA \to \mathcal B$ of
abelian sheaves on $\La$. 
By the above remark,
the functor of
sections over $G\times X$ is exact, and hence $\cA(G\times X) \to \mathcal B(G\times X)$ is an epimorphism
of abelian groups. From this we deduce that  $i_*\cA \to i_*\mathcal{B}$ is an epimorphism.
\end{proof}

\begin{cor}\label{cor:i*acy}
For any abelian sheaf $\cA$ on $\La$, we have
\[
H^i(G, i_*\cA) \cong 
\begin{cases}
\cA(*) &\text{if } i=0,\\
0 & \text{else.}
\end{cases}
\]
\end{cor}
\begin{proof}
Since the left adjoint $i^*$ is exact, $i_*$ sends injectives to injectives. Since $i_*$ is exact
and $\Gamma \circ i_*(-) = \Hom_{BG}(*, i_*(-)) \cong \Hom_{\Sh(\La)}(*, -)$ we see that
\[
H^i(G, i_*\cA) \cong 
H^i_{\Sh(\La)}(*, \cA) \cong
\begin{cases}
\cA(*) &\text{if } i=0,\\
0 & \text{else.} 
\end{cases} \qedhere
\]
\end{proof}

We let $E_{\dot}G$ be the simplicial manifold given in degree $p$ by $E_{p}G:=G^{p+1}$, and
$\phi^{*}\colon E_{q}G \to E_{p}G$, for $\phi\colon \{0<\dots<p\} \to \{0<\dots<q\}$, given by
$(g_{0}, \dots, g_{q}) \mapsto (g_{\phi(0)}, \dots, g_{\phi(p)}).$ The group $G$ acts on $E_{\dot}G$
via diagonal left multiplication. We denote the simplicial object of $BG$ given by $y(E_{\dot}G)$ equipped with  diagonal $y(G)$-action by $\l E_{\dot}G$.
%
For an abelian group object $\cA$ in $BG$, the degree-wise sections over $\l{E_{\dot}G}$ form a cosimplicial abelian group $\Gamma(\l{E_{\dot}G}, \cA)$.
\begin{prop}\label{prop:cochains-neu} 
Let $\cA$ be an abelian group object of $BG$. Then 
\[
H^{*}(G,\cA) \cong H^{*}(\Gamma(\l{E_{\dot}G}, \cA)) .
\]
\end{prop}
\begin{proof}
The projection $\l y(G) \to *$ is an epimorphism in $BG$. The \v{C}ech nerve of this morphism is precisely $\l{E_{\dot}G}$. We thus have a quasi-isomorphism
\[
\cA \cong \iHom_{BG}(*,\cA) \xrightarrow{\simeq} \iHom_{BG}(\l{E_{\dot}G}, \cA).
\]
Using Lemma \ref{lem:iHom} and Corollary \ref{cor:i*acy} we see that the complex on the right hand side consists of $\Gamma$-acyclic objects. 
We conclude using $\Gamma(\iHom_{BG}(\l{E_{\dot}G}, \cA)) = \Hom_{BG}(*, \iHom_{BG}(\l{E_{\dot}G}, \cA)) \cong \Hom_{BG}(\l{E_{\dot}G}, \cA) = \Gamma(\l{E_{\dot}G}, \cA)$.
\end{proof}

\section{Locally analytic representations}
\label{sec:2}

In Example \ref{ex:examples-of-sheaves}, we saw how to associate an abelian group object of $BG$ to any finite dimensional locally analytic representation of $G$.
The goal of this section is to extend this to arbitrary locally analytic representations 
in the sense of Schneider and Teitelbaum \cite{ST},
and to relate the cohomology groups thus defined to 
the cohomology groups defined in terms of analytic cochains (Proposition \ref{prop:lacochains}).

We begin by recalling some basic notions about analytic functions and locally analytic representations. References are \cite{Feaux} or \cite[\S\S 2,3]{ST}.

If $W$ is a Banach space, a function $f\colon X \to W$ from a manifold $X$ to $W$ is called analytic, if, in  local charts, it is given by convergent power series with coefficients in $W$. 
The reader who is only interested in representations on Banach spaces can skip the following general definition and also all ``admissibility'' considerations later on.
Let $V$ be a locally convex separated $K$-vector space. 
A \emph{BH-space} \footnote{\emph{B}anach-\emph{H}ausdorff} for $V$ is a continuous inclusion of a separated Banach space $W \hookrightarrow V$ \cite[\S 1.2]{Feaux}.
Let $X$ be a manifold. A function $f\colon
X \to V$ is called  \emph{analytic}, 
if every $x\in X$ admits a neighborhood $U$, and a BH-space $W\hookrightarrow V$ such that $f|_{U}$ factors through an analytic map $U \to W$.
 We denote the set of all analytic functions $X\to V$ by $C^{\an}(X,V)$.  By \cite[Kor.~2.4.4]{Feaux}, $C^{\an}(X,V)$  is  a module over the algebra of analytic functions $C^{\an}(X,K)$ .
 For varying $X$, this is a sheaf on $\La$ denoted by $C^{\an}(-,V)$.

A \emph{topological representation} of the Lie group $G$ on  $V$ is an action of $G$ on $V$ by continuous automorphisms, i.e., a homomorphism $G \to \Aut(V)$ to the abstract group of continuous automorphisms $\Aut(V)$ of $V$. 
The topological representation is called
\emph{locally analytic} if all orbit maps $G\to V, g\mapsto gv$, are analytic (this is called a weakly analytic representation in  \cite[Def.~3.1.5]{Feaux}). 

\begin{ex}\label{ex:Banach-space-reps}
Let $W$ be a Banach space. Then $\Aut(W)$ is an open subset of the Banach space of continuous endomorphisms $\operatorname{End}(W)$.
F\'eaux de Lacroix shows in \cite[Kor.~3.1.9]{Feaux} that a topological representation of $G$ on $W$ is locally analytic if and only if the corresponding homomorphism $r\colon G \to \Aut(W) \subset \operatorname{End}(W)$ is analytic. Assume that this is the case. Let $X$ be a manifold, and let $\rho\colon X \to G$ and $f\colon X \to W$ be analytic maps. Then the point-wise product $\rho f\colon X \to W, x\mapsto \rho(x)f(x),$ is again analytic. Indeed, $\rho f$  equals the composition
\[
X \xrightarrow{(\rho,\id_{X})} G\times X \xrightarrow{r\times f} \Aut(W) \times W \subseteq \operatorname{End}(W) \times W \xrightarrow{\mathrm{ev}} W,
\]
where the first two maps are analytic by assumption, and the evaluation map $\mathrm{ev}$ is continuous and bilinear. 
It follows, that we get an action of $y(G)$ on the sheaf $C^{\an}(-,W)$, and $C^{\an}(-,W)$ can naturally be viewed as an object of $BG$. This generalizes Example \ref{ex:examples-of-sheaves}.
\end{ex}
For a general locally analytic representation of $G$ on $V$, this need no longer be true. Instead, we have to consider the subsheaf of $C^{\an}(-,V)$ of admissible functions as we  explain in the next paragraph.  The example above shows that for representations on Banach spaces, every analytic function is admissible. 

Let $G\to \Aut(V)$ be a topological representation.
We call an analytic function $f\colon X \to V$ \emph{admissible}, if the map $\hat f\colon G\times X \to V, (g,x)\mapsto gf(x)$, is analytic. 
Note that $\hat f$ is analytic iff its restriction $\hat f|_{U\times X}$ for some open subset $U\subset G$ is analytic. Indeed, if this is the case, then for any $h\in G$ the restriction $\hat f|_{hU\times X}$ is equal to the composition 
\[
(hU \times X) \xrightarrow{(g,x)\mapsto (h^{-1}g,x)} U \times X \xrightarrow{\hat f|_{U\times X}} V \xrightarrow{h\cdot} V,
\]
where the first two maps are analytic, and the last one is continuous and linear.
We define 
\[
\ul V(X) := C^{\ad}(X,V) := \{f \in C^{\an}(X,V)\,|\,  f \text{ is admissible}\}.
\]
This is a $C^{\an}(X,K)$-submodule of $C^{\an}(X,V)$.
We claim that  $\ul V$ is a subsheaf of $C^{\an}(-,V)$ and that the point-wise multiplication by $G$-valued analytic maps defines an action of $y(G)$ on $\ul V$.
We will henceforth view $\ul V$ as an abelian group object of $BG$.
\begin{proof}[Proof of the claim]
If $f\in \ul V(X)$ and $\phi\colon Y \to X$ is an analytic map between manifolds, then $f\circ\phi$ is analytic. Moreover, $\widehat{f\circ\phi} = \hat f \circ (\id_{G}\times \phi)$ is analytic as well, hence $f\circ \phi$ is admissible. Thus $\ul V$ is a presheaf.
Since admissibility is a local property, $\ul V$ is a sheaf.

Now let $\rho\colon X \to G$ be an analytic map. We define $\rho f$ by $(\rho f)(x) := \rho(x)f(x)$. We have to show that $\rho f$ is analytic and admissible. But this is clear since $\rho f$ equals the composition $X \xrightarrow{(\rho,\id_{X})} G \times X \xrightarrow{\hat f} V$ and $\widehat{\rho f}$ equals the composition $G\times X \xrightarrow{(g,x) \mapsto (g\rho(x), x)} G \times X \xrightarrow{\hat f} V$.
\end{proof}

\begin{ex}
A topological representation of $G$ on $V$ is locally analytic if and only if $\ul V(*) = V$.
\end{ex}

\begin{dfn}
For a locally analytic representation of $G$ on $V$ and $i\geq 0$ we define the locally analytic group cohomology of $G$ with coefficients in $V$ as
\[
H^{i}_{\an}(G,V) := H^{i}(G, \ul V).
\]
\end{dfn}

Recall that a homogeneous analytic $p$-cochain of $G$ with coefficients in $V$ is an analytic function $f\colon E_pG\to V$ which is $G$-equivariant, i.e., satisfies $f(gg_0, \dots, gg_p) = gf(g_0, \dots g_p)$. We denote the complex of homogeneous analytic cochains by $C^{\an}_{G}(E_{\dot}G, V)$. Its differential is induced by the simplicial structure of $E_{\dot}G$.

\begin{prop}\label{prop:lacochains}
The cohomology $H^*_{\an}(G, V)$ is isomorphic to the cohomology of the complex of homogeneous analytic cochains
$C^{\an}_{G}(E_\dot G, V)$.
\end{prop}
\begin{proof}
By Proposition \ref{prop:cochains-neu} we have $H^*(G,\ul V) \cong H^*(\Gamma(\l{E_{\dot}G},\ul V))$. 
Using the Yoneda lemma we see that a section in $\Gamma(\l{E_pG}, \ul V)=\Hom_{BG}(\l{E_pG}, \ul V)$ is just an admissible function $f\colon E_pG \to V$ such that 
\[
\xymatrix{
G\times E_pG \ar[r]^-{\id\times f} \ar[d]_{\text{diagonal multiplication}} & G\times V \ar[d]^{\text{action}} \\
E_pG \ar[r]^f & V
}
\]
commutes, i.e., a $G$-equivariant admissible function $E_pG\to V$. 

To prove the claim, it suffices to show that, vice versa, every $G$-equivariant analytic function $f\colon E_pG \to V$ is admissible.
But, by the $G$-equivariance, $\hat f$ is equal to the composition $G \times E_{p}G \xrightarrow{(g,(g_0,\dots, g_p))\mapsto (gg_0, \dots, gg_p)} E_{p}G \xrightarrow{f} V$ and thus analytic. Thus $f$ is admissible. 
\end{proof}

\section{Differential forms and Lie algebra cohomology}
\label{sec:DifferentialFormsAndLiaAlgebraCohomology}

In this section, we introduce sheaves of differential forms with coefficients in a locally analytic representation $V$ of $G$ as objects of $BG$. Again, unless $V$ is a Banach space, we have to restrict to admissible forms. We  show that the Lie algebra cohomology of the $K$-Lie algebra $\lg$ of $G$ with coefficients in $V$ can be computed as the cohomology in $BG$ of the complex of forms on $G$ with coefficients in $V$.

Let $V$ be a separated locally convex $K$-vector space.  For a submersion $Y \to X$ we denote
by $\Omega^k(Y/X,V)$ the vector space of relative analytic $k$-forms with values in $V$. Here, a $k$-form
$\omega$ is called analytic if, for any $k$-tuple $\phi_1,\dots, \phi_k$ of local sections of the 
vertical tangent bundle $T(Y/X)$, the function $Y\to V, y\mapsto  \omega(y)(\phi_1(y), \dots, \phi_k(y))$ is
analytic.
It suffices to check this for the local sections given by a local frame of $T(Y/X)$. 
In particular, every $y\in Y$ admits a neighborhood $U$ and a BH-space $W \hookrightarrow V$ such that 
$\omega|_{U}$ is in the image of $\Omega^{k}(U/X,W) \hookrightarrow \Omega^{k}(U/X, V)$.
It follows that the exterior derivative $d\omega$ is a well-defined form in $\Omega^{k+1}(Y/X,V)$.
If $V$ is finite dimensional, this is the usual notion of analytic forms.

For a fixed manifold $Y$, we have a complex of sheaves $\ul{\Omega}^{*}(Y,V)$ on $\La$ defined by
\[
\ul{\Omega}^{*}(Y,V)(X) := \Omega^{*}(X\times Y/X,V).
\]
Let $V$ be a locally analytic representation of $G$. We would like to equip this complex with a $y(G)$-action. As for functions, we have to restrict to a subcomplex of admissible forms in order to do this. Again, one can ignore this, if one is only interested in the case that $V$ is a Banach space.

 A form $\omega \in \Omega^k(Y/X,V)$ is called \emph{admissible}, if the form $\hat\omega$ on $G\times Y/G\times X$ given by 
\[
\hat\omega(g,y)(\sx_{1}, \dots, \sx_k) := g\cdot(\omega(y)(\sx_1,\dots, \sx_k)), 
\]
where $g\in G, y\in Y, \sx_i\in T_{(g,y)}(G\times Y/G\times X)\cong T_y(Y/X)$, is analytic. 
Equivalently, $\omega$ is admissible iff for any $k$-tuple of local sections $\phi_1, \dots, \phi_k$ of $T(Y/K)$ as above the function $\omega(\phi_1,\dots, \phi_k)$ is admissible.
As in the case of functions, this is the case iff $\hat\omega|_{U\times Y}$ is analytic for some open subset $U \subseteq G$.
The admissible $k$-forms form a $C^{\an}(Y,K)$-submodule of $\Omega^k(Y/X,V)$, which we denote by $\Omega^k_{\ad}(Y/X,V)$. 
They are also stable under the exterior derivative: 
Let $\omega$ be an admissible form. Since $G$ acts on $V$ by continuous linear automorphisms we have $\widehat{d\omega} = d\hat\omega$, and this form is analytic.
Thus, the admissible forms $\Omega^{*}_{\ad}(Y/X,V)$ form a subcomplex of the de Rham complex $\Omega^{*}(Y/X,V)$. 
\begin{ex}
If $V$ is a Banach space, it follows from Example \ref{ex:Banach-space-reps} that any $V$-valued analytic $k$-form is admissible.  
\end{ex}

We now fix a manifold $Y$.
For an analytic map between manifolds $X' \to X$ the pull-back map $\Omega^{k}(X\times Y/X) \to \Omega^{k}(X'\times Y/X')$ preserves admissible forms. 
Since admissibility is a local condition, 
$\ul\Omega^{k}_{\ad}(Y,V)$, defined by
\[
\ul\Omega^k_\ad(Y,V)(X) := \Omega^k_\ad(X\times Y/X,V),
\]
is a subsheaf of $\ul{\Omega}^{k}(Y,V)$, and $\ul{\Omega}^{*}_{\ad}(Y,V) \subseteq \ul{\Omega}^{*}(Y,V)$ is a subcomplex.

We define an action of $y(G)$ on $\ul\Omega^k_\ad(G,V)$ as follows: Let $\rho\colon X \to G$ be an analytic map and $\omega\in \Omega^k_\ad(X \times G / X, V)$ an admissible form.
For any $h \in G$, left translation by $h^{-1}$ induces a map $T_{(x,g)}(X\times G/X) \to T_{(x,h^{-1}g)}(X\times G/X)$ written $\sx\mapsto h^{-1}\sx$. Using this notation 
we define $\rho\omega$ by the formula
\[
(\rho\omega)(x,g)(\sx_1,\dots, \sx_k) := \rho(x)\cdot \left(\omega(x, \rho(x)^{-1}g)(\rho(x)^{-1}\sx_1, \dots, \rho(x)^{-1}\sx_k)\right).
\]

\begin{lemma}
This is a well-defined $y(G)$-action.
\end{lemma}
\begin{proof}
We have to show that $\rho\omega$ is analytic and admissible. Consider the analytic maps $\hat\rho\colon G\times X \times G \to G \times X \times G, (h,x,g) \mapsto (h\rho(x), x, \rho(x)^{-1}g)$ and $i_{1}\colon X \times G \hookrightarrow G\times X \times G, (x,g)\mapsto (1,x,g)$. Then $\rho\omega = i_{1}^{*}\hat\rho^{*}\hat\omega$, hence $\rho\omega$ is analytic. Similarly, we have $\widehat{\rho\omega}=\hat\rho^{*}\hat\omega$, hence $\rho\omega$ is admissible.
\end{proof}

We thus consider $\ul\Omega^k_\ad(G,V)$ as an abelian group object in $BG$. We want to show that it is acyclic.
Write $\widetilde V := \Hom(\bigwedge\nolimits^k\lg, V)$. The adjoint action of $G$ on $\lg$ and the given action of $G$ on $V$ induce a natural $G$-action on $\widetilde V$. 
\begin{lemma}\label{lem:Omega-induced}
This representation of $G$ on $\widetilde{V}$ is locally analytic.
We have an isomorphism 
\[
\ul\Omega^k_\ad(G,V) \cong \iHom_{BG}(\l y(G), \ul{\widetilde V}) \cong i_*i^*\ul{\widetilde V}.
\]
\end{lemma}
\begin{proof}
Let $Y$ be a manifold.
We claim that a function $f\colon Y \to \widetilde{V}$ is admissible if and only if the function $f_{\sx}\colon Y\to V, y \mapsto f(y)(\sx),$ is admissible for every $\sx \in \bigwedge^k\lg$.
Taking $Y$ to be a point this implies the first assertion of the lemma.

To prove the claim, assume first that $f$ is admissible. We have $\widehat{f_{\sx}}(g,y) = g(f(y)(\sx)) = (gf(y))(g\sx) = \hat f(g,y)(g\sx)$. 
The function $\hat f$ is analytic by assumption and so is
$g\mapsto g\sx$. Since the evaluation $\widetilde{V} \times \bigwedge^k\lg \to V$ is continuous and bilinear, and since $\bigwedge^k\lg$ is finite dimensional, \cite[Satz 2.4.3]{Feaux} implies that $\widehat{f_{\sx}}$ is analytic.

To see the converse, let $\sx_1, \dots, \sx_N$ be a basis of $\bigwedge^k\lg$ and $\sx_1^*, \dots, \sx_N^*$ the dual basis of $(\bigwedge^k\lg)^\vee$. We can write $f$ as a sum $f = \sum_{i=1}^N f_{\sx_i}\sx_i^*$ with $f_{\sx_i}$ admissible. Then $\hat f(g,y) = \sum_i \widehat{f_{\sx_i}}(g,y) g(\sx_i^*)$ and by \loccit again, $\hat f$ is analytic. 

We now prove the second assertion of the lemma.
For any manifold $X$, right translations by elements of $G$ induce a trivialization of the vertical tangent
bundle $T(X\times G/X) \cong (X\times G)\times \lg$. This gives a natural isomorphism of vector spaces
\begin{equation*}\label{mar25neu}
\Omega^k(X\times G/X,V) \cong C^{\an}(X\times G, \widetilde V).
\end{equation*}
Using the above claim one sees that this isomorphism restricts to an isomorphism
\[
\Omega^k_{\ad}(X\times G/X,V) \cong C^{\ad}(X\times G, \widetilde{V}).
\]
Under this isomorphism, the $y(G)(X)$-action on the left-hand side corresponds to the action on the right-hand side induced by left translations on $X\times G$ and the action on $\widetilde{V}$ mentioned above.
Using the isomorphism $C^{\ad}(X\times G,\widetilde{V}) \cong \iHom_{BG}(\l y(G), \ul{\widetilde V})(X)$,
this gives the first isomorphism stated in the Lemma. The second follows immediately from Lemma \ref{lem:iHom}.
\end{proof}

\begin{cor}\label{cor:CohomForms}
We have
\[
H^i(G,\ul{\Omega}^k_\ad(G,V)) \cong 
\begin{cases}
 \Hom_{K}(\bigwedge^k\lg, V) & \text{ if } i=0,\\
0 & \text{ else.}
\end{cases}
\]
\end{cor}
\begin{proof}
By Lemma \ref{lem:Omega-induced} and Corollary \ref{cor:i*acy} the higher cohomology groups vanish, and 
\begin{equation}\label{eq:IdentificationInvariantForms}
H^0(G, \ul\Omega^k_{\ad}(G,V))\cong \ul{\widetilde V}(*) = \Hom_{K}(\bigwedge\nolimits^k\lg, V). \qedhere
\end{equation}
\end{proof}
Explicitly, this isomorphism is given by evaluating a form at $1\in G$.

The differential $d$ of the complex $\ul\Omega^*_\ad(G,V)$ is compatible with the $y(G)$-action. Hence we can view $\ul\Omega^*_\ad(G,V)$ as a complex in $BG$ and we can compute its hypercohomology.

We now assume that $V$ is barrelled, i.e., that every closed convex absorbing subset is open (see \cite[\S 6]{Schneider-NFA}). 
For example, any complete metrizable locally convex space, in particular any Banach space, is barrelled.
Differentiating the orbit maps $g\mapsto gv$ then induces an action of the Lie algebra $\lg$ on $V$ \cite[S\"atze 3.1.3, 3.1.7]{Feaux}.  
\begin{cor}\label{cor:deRhamLie}
We have natural isomorphisms
\[
H^i(G, \ul\Omega^*_\ad(G,V)) \cong H^{i}(\lg, V)
\]
where the right-hand side is  Lie algebra cohomology. 
\end{cor}
\begin{proof}
Corollary \ref{cor:CohomForms} gives an isomorphism
\[
H^i(G, \ul{\Omega}_{\ad}^*(G,V)) \cong H^{i}\left( \Hom_{K}(\bigwedge\nolimits^{*}\lg,V)\right),
\]
where the differential on $\Hom_{K}(\bigwedge\nolimits^{*}\lg,V)$ is induced from the de Rham differential via \eqref{eq:IdentificationInvariantForms}. This is precisely the Chevalley-Eilenberg complex computing Lie algebra cohomology.
\end{proof}

\section{Differential forms and locally analytic group cohomology}
\label{sec:proof}

 As before, we fix a locally analytic representation
$G\to \Aut(V)$. In this section we use the Poincar\'e lemma to compare the hypercohomology of the complex of $V$-valued admissible forms with locally analytic group cohomology, and we give the proof of the Theorem announced in the Introduction.

Fix a manifold $Y$. 
A function $f\colon Y\times X \to V$ will be called locally constant along $Y$, if, for every $(y,x)\in Y\times X$, there exist open neighborhoods $Y' \subseteq Y$ of $y$ and $X'\subseteq X$ of $x$ such that $f|_{Y'\times X'}$ factors through the projection $Y'\times X'\to X'$.
We define 
\[
\ul C^{\lc}_{\ad}(Y,V)(X) := \{f\in C^{\ad}(X\times Y,V)\,|\, f \text{ is locally constant along $Y$}\}.
\]
It is easy to see that $X\mapsto \ul C^{\lc}_{\ad}(Y,V)(X)$ defines a sheaf on $\La$.
\begin{prop}\label{prop:DeRhamResolution}
The inclusion in degree 0
\[
\ul C^{\lc}_{\ad}(Y,V) \to \ul\Omega^{*}_{\ad}(Y,V)
\]
is a quasi-isomorpism.
\end{prop}
If $V$ is a Banach space, this is just the Poincar\'e lemma, and its usual proof works. For general locally convex $V$, it is a little bit more complicated, since we have to prove admissibility of primitives.
\begin{proof}
The map clearly induces an isomorphism on $H^0$, and it remains to show that
$H^k(\ul{\Omega}_{\ad}^*(Y,V))=0$ for $k>0$.

Let $X$ be a manifold, and let $\omega$ be a closed form in $\Omega^k_\ad(X\times Y/X, V)$. We will show that there is an $\eta\in \Omega^{k-1}_\ad(X\times Y/X,V)$ such that $d\eta=\omega$. Since all manifolds are strictly paracompact, it is enough to construct such an $\eta$ locally on $X$ and $Y$ (see Remark \ref{rem:strictly-paracompact}).

The rest of the proof uses some results and notations from the Appendix. It can be skipped on first reading.
Since $d\hat\omega = \widehat{d\omega} = 0$, the form $\hat\omega\in\Omega^k(G \times X \times Y/G\times X,V)$ is  closed. 
Replacing $G$ be a small open neighborhood of 1 and using local charts, we may assume that there are multiradii $\delta\in \R^m_+, \epsilon\in \R^n_+$ such that $G\times X \cong B_\delta(0) \subset K^m, Y\cong B_\epsilon(0) \subset K^n$, and a BH-space $W\hookrightarrow V$ such that $\hat\omega$ is given by a power series in $F_\delta(\Omega^k_\epsilon(W))$ (cf.~\eqref{eq:ConvergentRelativeForms}).
Choose a multiradius $\epsilon' < \epsilon$. The homotopy operator $h\colon \Omega^k_\epsilon(W) \to \Omega^{k-1}_{\epsilon'}(W)$ given by  lemma \ref{lem:Poincare} induces an operator $h\colon F_\delta(\Omega^k_\epsilon(W)) \to F_\delta(\Omega^{k-1}_{\epsilon'}(W))$. 
We define $\widetilde{\eta}:=h(\hat\omega)$. Hence $\widetilde{\eta}$ represents a relative analytic $k-1$-form on $G\times X \times Y'/G\times X$ with an open subset $Y'\subset Y$. 
Since $\hat\omega$ is closed, we have $d\widetilde{\eta}=\hat\omega|_{G\times X\times Y'}$.

For $g\in G$, let $i_g\colon X\times Y \to G\times X \times Y$ (and similarly with $Y$ replaced by $Y'$) be the inclusion $(x,y)\mapsto (g,x,y)$. 
We set $\eta:= i_1^*\widetilde{\eta}$. Clearly, $d\eta = i_1^*d\widetilde{\eta} = i_1^*\hat\omega = \omega$. 
To prove  that $\eta$ is admissible, we show that $\hat\eta=\widetilde{\eta}$.
Let $\Phi_g\colon V\to V$ be the continuous automorphism given by the action of $g$. 
We have to check that $i_g^*\widetilde{\eta} = \Phi_g\circ\eta$.
By restriction, $\Phi_{g}$ induces a continuous isomorphism of BH-spaces $W \to g(W)$ (more precisely, we view $W$ as a linear subspace of $V$ and let $g(W)$ be its image under the action of $g\in G$ with Banach space structure induced from $W$ via the linear isomorphism $\Phi_{g}|_{W}\colon W \xrightarrow{\cong} g(W)$).
We have
\begin{align*}
i_g^*\widetilde{\eta} &= i_g^*(h(\hat\omega)) &&\text{(by definition)} \\
&= h(i_g^*\hat\omega) &&\text{(using \eqref{diag:PhiEval} with $\Phi=h\colon \Omega^{q}_{\epsilon}(W) \to \Omega^{q-1}_{\epsilon'}(W)$)} \\
&= h(\Phi_g\circ \omega) &&\text{(definition of $\hat\omega$)} \\
&= \Phi_g\circ h(\omega) &&\text{(Lemma \ref{lem:NaturalityIntegration} for $\Phi=\Phi_g\colon W \to g(W)$)} \\
&= \Phi_g\circ \eta &&\text{(since $h(\omega)=h(i_{1}^{*}\hat\omega)=i_{1}^{*}h(\hat\omega) = i_{1}^{*}\hat\eta=\eta$).} \qedhere
\end{align*}
\end{proof}

\begin{proof}[Proof of the Theorem]
The sheaf $\ul C^{\lc}_{\ad}(G,V)$ carries a natural $y(G)$-action induced by left translations on $G$ and the given action on $V$.
 By Proposition \ref{prop:DeRhamResolution} and Corollary \ref{cor:deRhamLie} we have  isomorphisms
\begin{equation}\label{eq:CohomIsos}
H^{*}(G,\ul C^{\lc}_{\ad}(G,V)) \cong H^{*}(G, \ul\Omega^{*}_{\ad}(G,V)) \cong H^{*}(\lg, V).
\end{equation}
As in the proof of  Proposition \ref{prop:lacochains}, Proposition \ref{prop:cochains-neu} implies that $H^{*}(G,\ul C^{\lc}_{\ad}(G,V))$ is the cohomology of the complex $C^{\lc}_{G}(G\times E_{\dot}G,V)$ of $G$-equivariant analytic functions $G\times E_{p}G\to V$ that are locally constant along the first factor.

Since the open subgroups $G' \subseteq G$ form  a fundamental system of neighborhoods of $1\in G$ (see \cite[Lemma
18.7]{Schneider}), we have an isomorphism
\[
\colim_{G'\subset G \text{ open}} C^{\an}_{G'}(E_{\dot}G', V) \cong 
\colim_{G'\subset G \text{ open}} C^{\lc}_{G'}(G'\times E_{\dot}G', V).
\]
Because taking the colimit over a directed system is exact, we see that 
\[
\colim_{G'\subset G} H^*_{\an}(G', V) \to \colim_{G'\subset G} H^*(G', \ul{C}^{\lc}(G',V))
\]
is an isomorphism. Since the isomorphisms \eqref{eq:CohomIsos} are compatible with the restriction to
open subgroups, the claim follows.
\end{proof}

There is an additional action of $G$ on $\ul C^{\lc}_{\ad}(G,V)$ and on $\ul\Omega^*_{\ad}(G,V)$ induced by right translations on $G$. This action is compatible with the given $y(G)$-action.
It induces a $G$-action on the cohomology groups. Via the isomorphism \eqref{eq:CohomIsos} this corresponds to the $G$-action on $H^*(\lg,V)$ induced by the adjoint action on $\lg$ and left multiplication on $V$.
\begin{cor}\label{cor:LieAlgcohomforcompactG}
If $G$ is compact, there is a natural isomorphism
\[
H^*_{\an}(G,V) \cong H^*(\lg, V)^G.
\]
\end{cor}
\begin{proof}
Since $G$ is compact, every open subgroup is of finite index and contains an open normal subgroup. 
If $X$ is a compact manifold, every function in $\ul C^{\lc}_{\ad}(G,V)(X)$ factors through $G/H\times X$ for some open normal subgroup $H \unlhd G$. Thus -- using the notation from the previous proof --
\[
\ul C^{\lc}_{G}(G\times E_\dot G, V) = \colim_{H\unlhd G \text{ open}} C^{\an}_{G}(G/H \times E_\dot G, V). 
\]
Since the colimit over a directed system is exact, this induces an isomorphism $H^*(\lg, V) \cong \colim_{H\unlhd G} H^*\left(C^{\an}_{G}(G/H \times E_\dot G, V)\right)$.
Since each quotient $G/H$ is finite, and taking invariants under a finite group is an exact functor on $K$-vector spaces with an action by that group, we get
\begin{align*}
H^*(\lg,V)^G &\cong \colim_{H\unlhd G} H^*\left(C^{\an}_{G}(G/H\times E_\dot G,V)\right)^{G/H} \\
&\cong \colim_{H\unlhd G} H^*(C^{\an}_{G}(G/H\times E_\dot G,V)^{G/H}) \\
&\cong \colim_{H\unlhd G} H^*(C^{\an}_{G}(E_\dot G,V)) \cong H^*_{\an}(G,V). \qedhere
\end{align*}
\end{proof}

\section{Explicit description of the comparison map}
\label{sec:explicit}

We want to describe an explicit map of complexes which induces the comparison map $H^{*}_{\an}(G,V) \to H^{*}(\lg,V)$.  
Recall that $H^{*}_{\an}(G,V)$ is  computed by the complex of homogeneous locally analytic cochains $C^{\an}_{G}(E_{\dot}G,V)$, and that 
$H^{*}(\lg,V)$ is computed by the complex of $G$-invariant admissible differential forms $\Omega^{*}_{\ad}(G,V)^{G}$.

For integers $p\geq 0$ and $0 \leq i \leq p$, we denote by $d_{i}$ the partial exterior derivative in the direction of the $(i+1)$-th factor of the product $E_{p}G= G^{p+1}$. We denote by $\Delta_{p}\colon G \to E_{p}G$ the diagonal map.
For $f\in C^{\an}(E_{p}G,V)$ we set
\[
\Psi(f) := \Delta_{p}^{*}(d_{1}d_{2}\dots d_{p}f) \in \Omega^{p}(G,V).
\]
\begin{prop}
The map $\Psi$ induces a morphism of complexes $C^{\an}_{G}(E_{\dot}G,V)\to \Omega^{*}_{\ad}(G;V)^{G}$, which on cohomology  groups agrees with the comparison map $H^{*}_{\an}(G,V) \to H^{*}(\lg, V)$.
\end{prop}
\begin{rem}
Let us consider the special case that $K$ is $\Q_{p}$ and $V$ is finite dimensional. 
We want to indicate how the method of \cite{Huber-Kings} allows one to compare our map with Lazard's one.
The space of functions $C^{\an}(E_{p}G,V)$ has topological generators of the form $f_{0}\otimes \dots \otimes f_{p} \otimes v$ with $f_{i}\in C^{\an}(G,K)$ and $v\in V$. For such a function we have
\[
\Psi(f_{0}\otimes \dots \otimes f_{p}\otimes v)= f_{0}df_{1}\wedge \dots\wedge df_{p} \otimes v,
\]
and its image in $\Hom(\bigwedge^{p}\lg, V)$ is given by $f_{0}(1)df_{1}(1)\wedge \dots \wedge df_{p}(1) \otimes v$.

There is another simplicial model $\tilde E_{\dot}G$ for the universal $G$-bundle (cf.~\cite[\S 4.4]{Huber-Kings}), given by $\tilde E_{p}G = E_{p}G$, but with face maps 
\[
\tilde\partial_{i}(g_{0}, \dots, g_{n}) = 
\begin{cases}
(g_{0}, \dots, g_{i-1}, g_{i}g_{i+1}, g_{i+2}, \dots, g_{p}) & \text{if } i=0, \dots p-1, \\
(g_{0}, \dots, g_{p-1}) &\text{if } i=p.
\end{cases}
\]
The $G$-action on $\tilde E_{\dot}G$ is given by left multiplication on the first factor. 
There is a natural $G$-equivariant isomorphism $\tilde E_{\dot}G\cong E_{\dot}G$.
Huber and Kings show that Lazard's isomorphisms (for $G$ small enough) is induced by the map
\[
\Phi\colon C^{\an}_{G}(\tilde E_{\dot}G,V) \to \Hom(\bigwedge^{\dot}\lg, V),  
\]
$\Phi(f_{0}\otimes \dots \otimes f_{p}\otimes v) =  f_{0}(1)df_{1}(1)\wedge \dots \wedge df_{p}(1) \otimes v$ (see \cite[Prop.~4.6.1]{Huber-Kings}; this is formulated in the case of trivial coefficients, but can easily be adapted to our setting).
 The argument of \cite[Thm.~4.7.1]{Huber-Kings} shows that the composition of $\Phi$ with the isomorphism $C^{\an}_{G}(E_{\dot}G,V) \cong C^{\an}_{G}(E_{\dot}G,V)$ is homotopic to $\Psi$, hence both maps agree on cohomology groups.
\end{rem}
\begin{proof}[Proof of the Proposition]
From the proof of Proposition \ref{prop:cochains-neu} we have the acyclic resolution
$\ul V \xrightarrow{\simeq} \iHom_{BG}(\l{E_{\dot}G}, \ul V)$. For a manifold $X$ we have
\[
\iHom_{BG}(\l{E_{\dot}G}, \ul V)(X) = C^{\ad}(X\times E_{\dot}G, V)
\]
with $y(G)$-action induced from left translations on $E_{\dot}G$ and the action on $V$.
We  define $\Psi\colon C^{\ad}(X \times E_{p}G, V) \to \Omega^{p}_{\ad}(X\times G/X,V)$ by the same formula as above. We claim that this gives a morphism of complexes  $\Psi\colon \iHom_{BG}(\l{E_{\dot}G}, \ul V)\to \ul{\Omega}^{*}_{\ad}(G,V)$ in $BG$. 
\begin{proof}[Proof of the claim.]
One checks without difficulty that $\Psi$ is equivariant for the $y(G)$-action. 
Now consider $f\in C^{\ad}(E_{p}G\times X,V)$. Recall the face maps $\del_{i}\colon E_{p+1}G\to E_{p}G, (g_{0}, \dots, g_{p+1}) \mapsto (g_{0}, \dots, \widehat{g_{i}}, \dots, g_{p+1})$. The differential of the complex $C^{\ad}(E_{\dot}G\times X, V)$ maps $f$ to
\[
\sum_{i=0}^{p+1} (-1)^{i}\del_{i}^{*}f.
\]
Since $\del_{i}^{*}f$ is constant along the $(i+1)$-th factor $G$, we have $d_{i}(\del_{i}^{*}f)=0$. Since the partial derivatives commute up to sign, it follows that  
\begin{align*}
\Psi(\sum_{i=0}^{p+1} (-1)^{i}\del_{i}^{*}f) &= \Psi(\del_{0}^{*}f) \\
&= \Delta_{p+1}^{*}(d_{1}\dots d_{p+1}(\del_{0}^{*}f)) \\
&= \Delta_{p+1}^{*}(\del_{0}^{*}(d_{0}\dots d_{p}f)) \\
&= \Delta_{p}^{*}(d_{0}\dots d_{p}f) \\
&= \Delta_{p}^{*}(d (d_{1}\dots d_{p}f)) \\
&= d(\Delta_{p}^{*}(d_{1}\dots d_{p}f)) \\
&= d(\Psi(f)).  \qedhere
\end{align*}
\end{proof}
We thus have a commutative diagram
\[
\xymatrix{
\ul V \ar[dr] \ar[r]^-{\simeq} & \iHom_{BG}(\l{E_{\dot}G}, \ul V) \ar[d]^{\Psi} \\
& \ul{\Omega}^{*}_{\ad}(G,V)
}
\]
where the complexes on the right-hand side consist of acyclic sheaves. The proposition now follows by taking global sections.
\end{proof}

\appendix
\section*{Appendix: The Poincar\'e lemma}

Let $W$ be a $K$-Banach space with norm $\|\,.\,\|$. For a multiradius $\epsilon=(\epsilon_1,\dots, \epsilon_n) \in \R_+^n$ we denote the space of $\epsilon$-convergent power series in $n$ variables $x=(x_1,\dots, x_n)$ with coefficients in $W$ by $F_\epsilon(W)$:
\[
F_\epsilon(W) := \left\{ \sum_{I\in \N_0^n} a_Ix^I\,|\, a_I\in W, \|a_I\|\epsilon^I \xrightarrow{I\to\infty} 0 \right\}
\]
Equipped with the norm $\|\sum_I a_Ix^I\|_\epsilon := \max_I \|a_I\|\epsilon^I$, this is again a Banach space.

Let $\Phi\colon W \to W'$ be a continuous linear map between Banach spaces. It induces a continuous linear map $F_\epsilon(W) \to F_\epsilon(W')$. Let $B_\epsilon(0) \subset K^n$ be the closed ball of radius $\epsilon$ around $0$. For any $x\in B_\epsilon(0)$ we have the evaluation at $x$,  written $i_x^*\colon F_\epsilon(W) \to W$ and similarly for $W'$. Since $\Phi$ is continuous the diagram
\begin{equation}
\begin{split}\label{diag:PhiEval}
\xymatrix{
F_\epsilon(W) \ar[r]^-\Phi \ar[d]_{i_x^*} & F_\epsilon(W') \ar[d]^{i_x^*} \\
W \ar[r]^-\Phi & W'
}
\end{split}
\end{equation}
commutes.

For $q\geq 0$ we denote by $\Omega^q_\epsilon(W)$ the space of $\epsilon$-convergent $W$-valued $q$-forms in $n$ variables:
\[
\Omega^q_\epsilon(W) := \bigwedge\nolimits_K^q(K^n)^\vee \otimes_K F_\epsilon(W).
\]
Since $\bigwedge\nolimits_K^q(K^n)^\vee$ is a finite dimensional $K$-vector space, this is again a Banach space. The usual differential defines a continuous linear map $d\colon \Omega^q_\epsilon(W) \to \Omega^{q+1}_\epsilon(W)$.

There is natural injection $\Omega_{\epsilon}^{q}(W) \hookrightarrow \Omega^{q}(B_{\epsilon}(0), W)$ into the space of locally analytic $W$-valued $q$-forms. It is compatible with the differential. More generally, if $\delta\in \R^{m}_{+}$ is a second multiradius, we can identify $\delta$-convergent power series with coefficients in $\Omega^{q}_{\epsilon}(W)$ with relative $W$-valued forms:
\begin{equation}\label{eq:ConvergentRelativeForms}
F_{\delta}(\Omega^{q}_{\epsilon}(W)) \hookrightarrow \Omega^{q}\left(B_{\delta}(0) \times B_{\epsilon}(0)/B_{\delta}(0), W\right).
\end{equation}
On the other hand, every relative $q$-form is in the image of \eqref{eq:ConvergentRelativeForms} after shrinking $\delta$ and $\epsilon$ appropriately.

Let $\epsilon'\in\R^n_+$ be a multiradius which is component-wise strictly smaller than $\epsilon$, written $\epsilon' < \epsilon$. 
There is a continuous inclusion $i\colon \Omega^q_\epsilon(W) \hookrightarrow \Omega^q_{\epsilon'}(W)$.
\begin{lemma}[Poincar\'e lemma]\label{lem:Poincare}
Let $\epsilon' <\epsilon$ and $q>0$. Then there exists a bounded linear map 
\[
h\colon \Omega^q_\epsilon(W) \to \Omega^{q-1}_{\epsilon'}(W)
\]
such that $d\circ h + h\circ d=i$.
\end{lemma}
\begin{proof}
We have
\[
\Omega^q_\epsilon(W) = \bigoplus_{1\leq k_{1}< \dots < k_{q}\leq n} F_\epsilon(W) dx_{k_1}\dots dx_{k_q}.
\]
Set $C:= \max_i ({\epsilon_i}/{\epsilon_i'})$. By assumption we have $C>1$. Hence, for integers $N \gg 0$, we have
$|1/(N+q)| \leq C^N$.
We define 
\[
h(x^I dx_{k_1}\dots dx_{k_q}) := \frac{1}{|I|+q}\sum_{\alpha=1}^q
(-1)^{\alpha-1}x^{I+e_{k_\alpha}} dx_{k_1}\dots \widehat{dx_{k_\alpha}} \dots dx_{k_q}.
\]
and 
\[
h\left(\sum a_{I}x^Idx_{k_1}\dots dx_{k_q}\right) := \sum a_{I} h(x^I dx_{k_1}\dots dx_{k_q}).
\]
Since
\begin{equation*}
\left\|\frac{a_{I}}{|I|+q}\right\| \epsilon'^I
\leq \|a_{I}\|C^{|I|} \epsilon'^I
\leq \|a_{I}\|\epsilon^I \text{ for }|I| \gg 0
\end{equation*}
it follows that the power series  $\sum_{I}\frac{a_{I}}{|I|+q}x^{I+e_{k_{\alpha}}}$ is
$\epsilon'$-convergent, whence that $h$ is well defined, and also that $h$ is a bounded linear operator.

By continuity, it is now enough to check the equality $dh+hd=i$ on  monomials $x^I
dx_{k_1}\dots dx_{k_q}$. Relabeling the coordinates, we may moreover assume that $(k_1, \dots, k_q)
= (1, \dots, q)$.
We have
\begin{multline*}
dh(x^I dx_{1}\dots dx_{q}) = d\left(\frac{1}{|I|+q}\sum_{\alpha=1}^q
(-1)^{\alpha-1}x^{I+e_\alpha} dx_{1}\dots \widehat{dx_{\alpha}} \dots dx_{q}\right) \\
= \left(\frac{1}{|I|+q} \sum_{\alpha=1}^{q} (i_{\alpha}+1)x^{I}dx_{1}\dots dx_{q}\right) + \\
	\frac{1}{|I|+q}\sum_{\alpha=1}^{q}\sum_{\beta=q+1}^{n}(-1)^{\alpha-1}(-1)^{q-1}i_{\beta}x^{I+e_{\alpha}-e_{\beta}}
		dx_{1}\dots \widehat{dx_{\alpha}} \dots dx_{q} dx_{\beta}\\
= \frac{(\sum_{\alpha=1}^{q}i_{\alpha})+q}{|I|+q} x^I dx_{1}\dots dx_{q} +  \\
	\frac{1}{|I|+q}\sum_{\alpha=1}^{q}\sum_{\beta=q+1}^{n}(-1)^{\alpha+q}i_{\beta}x^{I+e_{\alpha}-e_{\beta}}
		dx_{1}\dots \widehat{dx_{\alpha}} \dots dx_{q} dx_{\beta}	
\end{multline*}
and 
\begin{multline*}
hd(x^{I}dx_{1}\dots dx_{q}) = h\left((-1)^{q}\sum_{\beta=q+1}^{n}  i_{\beta}x^{I-e_{\beta}}dx_{1}\dots dx_{q}dx_{\beta}\right) \\
= \frac{(-1)^{q}}{|I|+q}\sum_{\alpha=1}^{q} \sum_{\beta=q+1}^{n}  (-1)^{\alpha-1} i_{\beta}x^{I+e_{\alpha}-e_{\beta}} dx_{1}\dots \widehat{dx_{\alpha}} \dots dx_{q} dx_{\beta}	 + \\
	\frac{(-1)^{q}}{|I|+q} \sum_{\beta=q+1}^{n}  (-1)^{q}i_{\beta}x^{I}dx_{1}\dots dx_{q} \\
= \frac{1}{|I|+q}\sum_{\alpha=1}^{q}\sum_{\beta=q+1}^{n}(-1)^{\alpha+q-1}i_{\beta}x^{I+e_{\alpha}-e_{\beta}}
		dx_{1}\dots \widehat{dx_{\alpha}} \dots dx_{q} dx_{\beta}	+ \\
		\frac{(\sum_{\beta=q+1}^{n}i_{\beta})}{|I|+q} x^I dx_{1}\dots dx_{q} \\
\end{multline*}
Thus, $(dh+hd)(x^{I}dx_{1}\dots dx_{k}) = x^{I}dx_{1}\dots dx_{k}$. This finishes the proof of the
lemma.
\end{proof}

\begin{lemma}\label{lem:NaturalityIntegration}
Let $\Phi\colon W \to W'$ be a bounded linear map between Banach spaces. It induces a map $\Omega^q_\epsilon(W) \to \Omega^q_\epsilon(W')$, denoted by the same symbol. For $q>0$ and $\epsilon' < \epsilon$, the diagram
\[
\xymatrix{
\Omega^q_\epsilon(W) \ar[r]^-h \ar[d]_\Phi & \Omega^{q-1}_{\epsilon'}(W) \ar[d]^\Phi \\
\Omega^q_\epsilon(W') \ar[r]^-h  & \Omega^{q-1}_{\epsilon'}(W')
}
\]
commutes.
\end{lemma}
\begin{proof}
This follows directly from the definitions.
\end{proof}

\bibliographystyle{amsalpha}
\bibliography{Lazard}

\providecommand{\bysame}{\leavevmode\hbox to3em{\hrulefill}\thinspace}
\providecommand{\MR}{\relax\ifhmode\unskip\space\fi MR }
\providecommand{\MRhref}[2]{%
  \href{http://www.ams.org/mathscinet-getitem?mr=#1}{#2}
}
\providecommand{\href}[2]{#2}
\begin{thebibliography}{HKN11}

\bibitem[FdL99]{Feaux}
Christian~Tobias F{\'e}aux~de Lacroix, \emph{Einige {R}esultate \"uber die
  topologischen {D}arstellungen {$p$}-adischer {L}iegruppen auf unendlich
  dimensionalen {V}ektorr\"aumen \"uber einem {$p$}-adischen {K}\"orper},
  Schriftenreihe Math. Inst. Univ. M\"unster 3. Ser., vol.~23, Univ. M\"unster,
  1999, pp.~x+111. \MR{1691735 (2000k:22021)}

\bibitem[Fla08]{Flach}
Matthias Flach, \emph{Cohomology of topological groups with applications to the
  {W}eil group}, Compos. Math. \textbf{144} (2008), no.~3, 633--656.
  \MR{2422342 (2009f:14033)}

\bibitem[HK11]{Huber-Kings}
Annette Huber and Guido Kings, \emph{A {$p$}-adic analogue of the {B}orel
  regulator and the {B}loch-{K}ato exponential map}, J. Inst. Math. Jussieu
  \textbf{10} (2011), no.~1, 149--190. \MR{2749574 (2012a:19010)}

\bibitem[HKN11]{HKN}
Annette Huber, Guido Kings, and Niko Naumann, \emph{Some complements to the
  {L}azard isomorphism}, Compos. Math. \textbf{147} (2011), no.~1, 235--262.
  \MR{2771131 (2012d:22016)}

\bibitem[Koh11]{Kohlhaase}
Jan Kohlhaase, \emph{The cohomology of locally analytic representations}, J.
  Reine Angew. Math. \textbf{651} (2011), 187--240. \MR{2774315}

\bibitem[Laz65]{Lazard}
Michel Lazard, \emph{Groupes analytiques {$p$}-adiques}, Inst. Hautes \'Etudes
  Sci. Publ. Math. (1965), no.~26, 389--603. \MR{0209286 (35 \#188)}

\bibitem[Lec12]{Lechner}
Sabine Lechner, \emph{{A comparison of locally analytic group cohomology and
  Lie algebra cohomology for p-adic Lie groups}},
  \href{http://arxiv.org/abs/1201.4550}{arXiv:1201.4550}, 2012.

\bibitem[Sch02]{Schneider-NFA}
Peter Schneider, \emph{Nonarchimedean functional analysis}, Springer Monographs
  in Mathematics, Springer-Verlag, Berlin, 2002. \MR{1869547 (2003a:46106)}

\bibitem[Sch11]{Schneider}
\bysame, \emph{{$p$}-adic {L}ie groups}, Grundlehren der Mathematischen
  Wissenschaften [Fundamental Principles of Mathematical Sciences], vol. 344,
  Springer, Heidelberg, 2011. \MR{2810332 (2012h:22010)}

\bibitem[SGA72]{SGA41}
\emph{Th\'eorie des topos et cohomologie \'etale des sch\'emas. {T}ome 1:
  {T}h\'eorie des topos}, Lecture Notes in Mathematics, Vol. 269,
  Springer-Verlag, Berlin-New York, 1972, S{\'e}minaire de G{\'e}om{\'e}trie
  Alg{\'e}brique du Bois-Marie 1963--1964 (SGA 4), Dirig{\'e} par M. Artin, A.
  Grothendieck, et J. L. Verdier. Avec la collaboration de N. Bourbaki, P.
  Deligne et B. Saint-Donat. \MR{0354652 (50 \#7130)}

\bibitem[ST02]{ST}
Peter Schneider and Jeremy Teitelbaum, \emph{Locally analytic distributions and
  {$p$}-adic representation theory, with applications to {${\rm GL}_2$}}, J.
  Amer. Math. Soc. \textbf{15} (2002), no.~2, 443--468 (electronic).
  \MR{1887640 (2003b:11132)}

\end{thebibliography}

\end{document}